\documentclass[12pt,english]{amsart}
\usepackage{lmodern}
\usepackage{helvet}

\usepackage[T1]{fontenc}
\usepackage[latin9]{inputenc}
\usepackage{mathrsfs}
\usepackage{amsthm}
\usepackage{amstext}
\usepackage{amssymb}
\usepackage{esint}

\makeatletter

\newcommand{\noun}[1]{\textsc{#1}}

\numberwithin{equation}{section}
\numberwithin{figure}{section}
\theoremstyle{plain}
\newtheorem{thm}{\protect\theoremname}[section]
  \theoremstyle{plain}
  \newtheorem{lem}[thm]{\protect\lemmaname}
  \theoremstyle{plain}
  \newtheorem{prop}[thm]{\protect\propositionname}
  \theoremstyle{remark}
  \newtheorem{rem}[thm]{\protect\remarkname}
  \theoremstyle{plain}
  \newtheorem{cor}[thm]{\protect\corollaryname}

\usepackage{amsfonts}
\usepackage{amssymb}
\usepackage[all]{xy}
\usepackage{array}
\usepackage{lmodern}
\usepackage[T1]{fontenc}

\def\QQ{\mathbb{Q}}
\def\RR{\mathbb{R}}
\def\CC{\mathbb{C}}

\def\FF{\mathbb{F}}
\def\ZZ{\mathbb{Z}}
\def\PP{\mathbb{P}}
\def\TT{\mathbb{T}}
\def\GG{\mathbb{G}}
\def\uz{\underline{z}}

\def\W{\mathcal{W}}

\def\Z{\mathcal{Z}}
\def\CW{\mathscr{W}}
\def\CZ{\mathscr{Z}}
\def\CV{\mathscr{V}}
\def\CU{\mathscr{U}}
\def\CY{\mathscr{Y}}
\def\CX{\mathscr{X}}
\def\T{\mathcal{T}}
\def\Y{\mathcal{Y}}

\def\M{\mathcal{M}}

\def\D{\mathcal{D}}

\def\k12{\mathcal{K}_{\lambda_1,\lambda_2}}
\def\tk12{\tilde{\mathcal{K}}_{\lambda_1,\lambda_2}}
\def\ck12{\check{\mathcal{K}}_{\lambda_1,\lambda_2}}
\def\Res{\text{Res}}
\def\MHS{\mathrm{MHS}}

\def\ay{\mathbf{i}}

\theoremstyle{definition}

\theoremstyle{definition}

\theoremstyle{theorem}

\theoremstyle{theorem}

\theoremstyle{theorem}

\theoremstyle{definition}

\makeatother

\usepackage{babel}
  \providecommand{\corollaryname}{Corollary}
  \providecommand{\lemmaname}{Lemma}
  \providecommand{\propositionname}{Proposition}
  \providecommand{\remarkname}{Remark}
\providecommand{\theoremname}{Theorem}

\begin{document}

\title[Explicit higher Chow cycles]{An explicit basis for the rational higher Chow groups of abelian
number fields}

\author{Matt Kerr and Yu Yang}

\subjclass[2000]{14C25, 14C30, 19E15}
\begin{abstract}
We review and simplify A. Beilinson's construction of a basis for
the motivic cohomology of a point over a cyclotomic field, then promote
the basis elements to higher Chow cycles and evaluate the KLM regulator
map on them.
\end{abstract}
\maketitle

\section{Introduction}

Let $\zeta_{N}\in\CC^{*}$ be a primitive $N^{\text{th}}$ root of
$1$ ($N\geq2$). The seminal article \cite{Be1} of A. Beilinson
concludes with a construction of elements $\Xi_{b}$ ($b\in(\ZZ/N\ZZ)^{*}$)
in motivic cohomology
\[
H_{\M}^{1}\left(\text{Spec}(\QQ(\zeta_{N})),\QQ(n)\right)\cong K_{2n-1}^{(n)}\left(\QQ(\zeta_{N})\right)\otimes\QQ
\]
mapping to $Li_{n}(\zeta_{N}^{b})=\sum_{k\geq1}\tfrac{\zeta_{N}^{kb}}{k^{n}}\in\CC/(2\pi\ay)^{n}\RR$
under his regulator. Since by Borel's theorem \cite{Bo1} $\text{rk}K_{2n-1}^{(n)}(\QQ(\zeta_{N}))_{\QQ}=\tfrac{1}{2}\phi(N)$
(for $N\geq3$), an immediate consequence is that the $\{\Xi_{b}\}$
span $K_{2n-1}^{(n)}(\QQ(\zeta_{N}))_{\QQ}$; indeed, Beilinson's
results anticipated the eventual proofs \cite{Ra,Bu} of the equality
(for number fields) of his regulator with that of Borel \cite{Bo2}.
An expanded account of his construction was written up by Neukirch
(with Rapoport and Schneider) in \cite{Ne}, up to a ``crucial lemma''
((2.4) in {[}op. cit.{]}) required for the regulator computation,
which was subsequently proved by Esnault \cite{Es}.

The intervening years have seen some improvements in technology, especially
Bloch's introduction of higher Chow groups \cite{Bl}, which yield
an \emph{integral} definition of motivic cohomology for smooth schemes
$X$. In particular, we have%
\footnote{We use the shorthand $CH^{*}(F,*)$ ($Z^{*}(F,*)$, etc.) for $CH^{*}(\text{Spec}(F),*)$
($F$ a field).%
}\begin{flalign*}H_{\M}^{1}\left(\text{Spec}(\QQ(\zeta_{N})),\ZZ(n)\right) & \cong CH^{n}\left(\QQ(\zeta_{N}),2n-1\right) \\
& :=H_{2n-1}\left\{ Z^{n}(\QQ(\zeta_{N}),\bullet),\partial\right\} ,
\end{flalign*}and can ask for explicit cycles in $\ker(\partial)\subset Z^{n}\left(\QQ(\zeta_{N}),2n-1\right)$
representing (multiples of) Beilinson's elements $\Xi_{b}$. Another
relevant development was the explicit realization of Beilinson's regulator
in \cite{KLM,KL} as a morphism $\widetilde{AJ}$ of complexes, from
a \emph{rationally} quasi-isomorphic subcomplex $Z_{\RR}^{n}(X,\bullet)$
of $Z^{n}(X,\bullet)$ to a complex computing the absolute Hodge cohomology
of $X$. Here this ``KLM morphism'' yields an Abel-Jacobi mapping\begin{equation}\label{eqIn1}AJ:\, CH^n\left( \QQ(\zeta_N ),2n-1\right) \otimes \QQ \to \CC/(2\pi \ay)^n \QQ ,
\end{equation}and in the present note we shall construct (for all $n$) higher Chow
cycles
\[
\hat{\mathscr{Z}}_{b}\in\ker(\partial)\subset Z_{\RR}^{n}(\QQ(\zeta_{N}),2n-1)\otimes\QQ
\]
with 
\[
(n-3)N^{n-1}\hat{\mathscr{Z}}_{b}\in Z_{\RR}^{n}\left(\QQ(\zeta_{N}),2n-1)\right)\;\;\;\text{and}\;\;\; AJ(\hat{\mathscr{Z}}_{b})=Li_{n}(\zeta_{N}^{b}).
\]
(See Theorems \ref{thm1}, \ref{thm2}, and \ref{thm3}, with $\hat{\mathscr{Z}}=\tfrac{(-1)^{n}}{N^{n-1}}\tilde{\mathscr{Z}}$.)
This is entirely more explicit than the constructions in \cite{Be1,Ne},
and yields a brief and transparent evaluation of the regulator, which
moreover allows us to dispense with some of the hypotheses of \cite[Lemma 2.4]{Ne}
or \cite[Theorem 3.9]{Es} and thus avoid the more complicated \inputencoding{latin1}{construction
of \cite[Lemma 3.1]{Ne}.}\inputencoding{latin9} Furthermore, in concert
with the anticipated extension of $\widetilde{AJ}$ to the entire
complex $Z^{n}(X,\bullet)$ (making \eqref{eqIn1} integral), we expect
that our cycles will be useful for studying the torsion in $CH^{n}\left(\QQ(\zeta_{N}),2n-1\right)$
as begun in \cite{Pe,Pe2}, cf. Remark \ref{rem4.2-1} and $\S$\ref{torsion}.

\subsubsection*{Acknowledgments:}

The authors gratefully acknowledge support from NSF Grant DMS-1361147,
and helpful remarks from the referee and J. Fresan.

\section{Beilinson's construction}

In this section we show that (the graph of) the $n$-tuple of functions
$\{1-\zeta_{N}z_{1}\cdots z_{n-1},\left(\tfrac{z_{1}}{z_{1}-1}\right)^{N},\ldots,\left(\tfrac{z_{n-1}}{z_{n-1}-1}\right)^{N}\}$
completes to a relative motivic cohomology class on $(\square^{n-1},\partial\square^{n-1})$.
Most of the work that follows is to show that its image under a residue
map vanishes, cf. diagram \eqref{eqn2.5-3}. It also serves to establish
notation for \inputencoding{latin1}{$\S$}\inputencoding{latin9}\ref{sec3},
where we recast this class as a higher Chow cycle and compute its
regulator.

\subsection{Notation}

Let $N\geq2$, and $\zeta\in\CC$ be a primitive $N^{\text{th}}$
root of unity; i.e., $\zeta=e^{\frac{2\pi\ay a}{N}}$, where $a$
is coprime to $N$. Denoting by $\Phi_{N}(x)$ the $N^{\text{th }}$
cyclotomic polynomial, each such $a$ yields an embedding $\sigma$
of $\FF:=\tfrac{\QQ[\omega]}{(\Phi_{N}(\omega))}$ into $\CC$ (by
sending $\omega\mapsto\zeta$). (If $N=2$, then $\FF=\QQ$ and $\omega=\zeta=-1$.)

Working over any subfield of $\CC$ containing $\zeta$, write
\[
\square^{n}:=\left(\PP^{1}\backslash\{1\}\right)^{n}\supset\left(\PP^{1}\backslash\{0,1\}\right)^{n}=:\TT^{n},
\]
with coordinates $(z_{1},\ldots,z_{n})$. We have isomorphisms from
$\TT^{n}$ to $\GG_{m}^{n}$ (with coordinates $(t_{1},\ldots,t_{n})$),
given by $t_{i}:=\tfrac{z_{i}}{z_{i}-1}$. Define a function $f_{n}(\uz):=1-\zeta^{b}t_{1}\cdots t_{n}$
on $\TT^{n}$ (with $b$ coprime to $N$), and normal crossing subschemes
\[
S^{n}:=\left\{ \left.\uz\in\TT^{n}\:\right|\,\text{some }z_{i}=\infty\right\} \subset S^{n}\cup\left|(f_{n})_{0}\right|=:\tilde{S}^{n}\subset\TT^{n}.
\]
(Alternatively, we may view these schemes as defined over $\FF$ by
replacing $\zeta^{b}$ with $\omega^{b}$.)

Now consider the morphism\begin{flalign*}\imath_n : & \;\;\;\;\;\;\;\;\; \TT^{n-1}  \;\;\;\;\; \longrightarrow \;\;\;\;\;\TT^n \\
& \left( t_1,\ldots ,t_{n-1} \right) \; \longmapsto  \left( t_1,\ldots ,t_{n-1},(\zeta^b t_1 \cdots t_{n-1})^{-1} \right) .
\end{flalign*}We record the following:
\begin{lem}
\label{lem0}$\imath_{n}$ sends $\TT^{n-1}$ isomorphically onto
$\left|(f_{n})_{0}\right|$, with $\imath_{n}(\tilde{S}^{n-1})=\left|(f_{n})_{0}\right|\cap S^{n}$.
\end{lem}
We also remark that the Zariski closure of $\imath_{n}(\TT^{n-1})$
in $\square^{n}$ is just $\imath_{n}(\TT^{n-1})$.

\subsection{Results for Betti cohomology}

The construction just described has quite pleasant cohomological properties,
as we shall now see.
\begin{lem}
\label{lem1}As a $\QQ$-MHS, $H^{q}\left(\TT^{n},S^{n}\right)\cong\left\{ \begin{array}{cc}
\QQ(-n) & ,\; q=n\\
0 & ,\; q\neq n
\end{array}.\right.$\end{lem}
\begin{proof}
Apply the K\"unneth formula to $(\TT^{n},S^{n})\cong(\GG_{m},\{1\})^{n}$.\end{proof}
\begin{lem}
\label{lem2}As a $\QQ$-MHS, 
\[
H^{q}\left(\TT^{n},\tilde{S}^{n}\right)\cong\left\{ \begin{array}{cc}
\QQ(0)\oplus\QQ(-1)\oplus\cdots\oplus\QQ(-n) & ,\; q=n\\
0 & ,\; q\neq n
\end{array}.\right.
\]
\end{lem}
\begin{proof}
This is clear for $(\TT^{1},\tilde{S}^{1})\cong(\GG_{m},\{1,\bar{\zeta}\})$.
Now consider the exact sequence\begin{multline*}H^{*-1}(\TT^n,S^n)\overset{\imath^*_n}{\to} \underline{H^{*-1}(\TT^{n-1},\tilde{S}^{n-1})} \overset{\delta}{\to} \\
H^*(\TT^n,\tilde{S}^n) \to \underline{H^*(\TT^n,S^n)} \overset{\imath^*_n}{\to} H^*(\TT^{n-1},\tilde{S}^{n-1})
\end{multline*}of $\QQ$-MHS, associated to the inclusion $(\TT^{n-1},\tilde{S}^{n-1})\overset{\imath_{n}}{\hookrightarrow}(\TT^{n},S^{n})$.
(This is just the relative cohomology sequence, once one notes that
the pair $\left((\TT^{n},S^{n}),\imath_{n}(\TT^{n-1},\tilde{S}^{n-1})\right)=\left(\TT^{n},S^{n}\cup\imath_{n}(\TT^{n-1})\right)=(\TT^{n},\tilde{S}^{n})$
by Lemma \ref{lem0}.) If $*\neq n$, then the underlined terms are
$0$ via Lemma \ref{lem1} and induction. If $*=n$, then the \emph{end}
terms are $0$ via Lemma \ref{lem1} and induction, and \begin{equation}\label{eqn2.2-1}0\to H^{n-1}(\TT^{n-1},\tilde{S}^{n-1})\overset{\delta}{\to} H^n(\TT^n,\tilde{S}^n) \to H^n(\TT^n,S^n)\to 0
\end{equation}is a short-exact sequence.

Now observe that:\begin{itemize}\item{$H^n(\TT^n,S^n;\CC) = F^n H^n(\TT^n,S^n;\CC)$ is generated by the holomorphic form $\eta:=\tfrac{1}{(2\pi\ay)^n} \tfrac{dt_1}{t_1}\wedge\cdots\wedge\tfrac{dt_n}{t_n};$}\item{$H_{n-1}(\TT^{n-1},\tilde{S}^{n-1};\QQ)$ is generated by images $\underline{e}(U_i)$ of the cells $\bigcup_{i=0}^n U_i = [0,1]^n\setminus \bigcup_{\ell=1}^{n} \left\{ \sum x_i = \ell-\tfrac{a}{N} \right\} , $ where $\underline{e}: [0,1]^n \to \TT^n$ is defined by $(x_1,\ldots,x_n)\mapsto (e^{2\pi \ay x_1},\ldots ,e^{2\pi \ay x_n} ) = (t_1,\ldots, t_n)$; } 
\end{itemize} and\begin{itemize} \item{$\int_{\underline{e}(U_i)} \eta = \int_{U_i} dx_1 \wedge \cdots \wedge dx_n \in \QQ .$ }\end{itemize}(Writing
$\mathscr{S}^{1}$ for the unit circle, $\left((\mathscr{S}^{1})^{n},(\mathscr{S}^{1})^{n}\cap\tilde{S}^{n}\right)$
is a deformation retract of $(\TT^{n},\tilde{S}^{n})$. The $\underline{e}(U_{i})$
visibly yield all the relative cycles in the former, justifying the
second observation.) Together these immediately imply that \eqref{eqn2.2-1}
is split, completing the proof.
\end{proof}

\subsection{Results for Deligne cohomology}

Recall that Beilinson's absolute Hodge cohomology \cite{Be2} of an
analytic scheme $\mathrm{Y}$ over $\CC$ sits in an exact sequence\begin{multline*}0\to Ext^1_{\MHS}\left( \QQ(0),H^{r-1}(\mathrm{Y},\mathbb{A}(p))\right) \to H^r_{\D}\left( \mathrm{Y},\mathbb{A}(p)\right)  \\
\to Hom_{\MHS}\left( \QQ(0),H^r(\mathrm{Y} ,\mathbb{A}(p))\right) \to 0.
\end{multline*}(Here we use a subscript ``$\D$'' since the construction after
all is a ``weight-corrected'' version of Deligne cohomology; the
subscript ``$\MHS$'' of course means ``$\mathbb{A}$-$\MHS$''.)
We shall not have any use for details of its construction here, and
refer the reader to \cite[$\S$2]{KL}.
\begin{lem}
\label{lem3}The map $\imath_{n}^{*}:H_{\D}^{n}(\TT^{n},S^{n};\mathbb{A}(n))\to H_{\D}^{n}(\TT^{n-1},\tilde{S}^{n-1};\mathbb{A}(n))$
is zero $(\mathbb{A}=\QQ\text{ or }\RR)$.\end{lem}
\begin{proof}
Consider the exact sequence
\[
\to H_{\D}^{n}(\TT^{n},S^{n};\QQ(n))\overset{\imath_{n}^{*}}{\to}H_{\D}^{n}(\TT^{n-1},\tilde{S}^{n-1};\QQ(n))\overset{\delta_{\D}}{\to}H_{\D}^{n+1}(\TT^{n},\tilde{S}^{n};\QQ(n))\to.
\]
It suffices to show that $\delta_{\D}$ is injective. Now\begin{flalign*}
Hom_{\MHS}\left(\QQ(0),H^n(\TT^{n-1},\tilde{S}^{n-1};\QQ(n))\right) & = \{0\}\\
Hom_{\MHS}\left(\QQ(0),H^{n+1}(\TT^n,\tilde{S}^n;\QQ(n))\right) & = \{0\}
\end{flalign*}by Lemma \ref{lem2}, and so $\delta_{\D}$ is given by\small
\[
Ext_{\MHS}^{1}\left(\QQ(0),H^{n-1}(\TT^{n-1},\tilde{S}^{n-1};\QQ(n))\right)\overset{\delta_{\D}}{\to}Ext_{\MHS}^{1}\left(\QQ(0),H^{n}(\TT^{n},\tilde{S}^{n};\QQ(n))\right).
\]
\normalsize Since \eqref{eqn2.2-1} is split, the corresponding sequence
of $Ext^{1}$-groups is exact, and $\delta_{\D}$ is injective.
\end{proof}

\subsection{Results for motivic cohomology}

Let $X$ be any smooth simplicial scheme (of finite type), defined
over a subfield of $\CC$. We have Deligne class maps ($\mathbb{A}=\QQ$
or $\RR$)
\[
c_{\D,\mathbb{A}}:\, H_{\M}^{r}(X,\QQ(p))\to H_{\D}^{r}(X_{\CC}^{\text{an}},\mathbb{A}(p)).
\]
The case of particular interest here is where $r=1$, $X$ is a point,
and\begin{equation}\label{eqn2.4-1}c_{\D,\mathbb{A}}(\mathrm{Z})=\tfrac{1}{(2\pi \ay)^{p-1}} \int_{\mathrm{Z}^{\text{an}}_{\CC}} R_{2p-1} \in \CC/\mathbb{A}(p) \, ,
\end{equation}where (interpreting $\log(z)$ as the $0$-current with branch cut
along $T_{z}:=z^{-1}(\RR_{-})$)\begin{multline}\label{eqn2.4-2}R_{2p-1} := \sum_{k=1}^{2p-1} (2\pi\ay)^{k-1} R^{(k)}_{2p-1} \\
:= \sum_{k=1}^{2p-1} (2\pi\ay)^{k-1} \log(z_k)\tfrac{dz_{k+1}}{z_{k+1}}\wedge \cdots \wedge \tfrac{dz_{2p-1}}{z_{2p-1}} \cdot \delta_{T_{z_1} \cap \cdots \cap T_{z_{k-1}}}
\end{multline}is the regulator current of \cite{KLM,KL}, belonging to $D^{2p-2}\left((\PP^{1})^{\times(2p-1)}\right)$.
Here it is essential that the representative higher Chow cycle $\mathrm{Z}$
belong to the quasi-isomorphic subcomplex $Z_{\RR}^{p}(\text{pt.},\bullet)_{\QQ}\subset Z^{p}(\text{pt.},\bullet)_{\QQ}$
comprising cycles in good position with respect to certain real analytic
chains, cf. \cite[$\S$ 8]{KL} or Remark \ref{rem3.1-1} below.

Now take a number field $K$, $[K:\QQ]=d=r_{1}+2r_{2}$, and set 
\[
d_{m}=d_{m}(K):=\left\{ \begin{array}{cc}
r_{1}+r_{2}-1 & ,\; m=1\\
r_{1}+r_{2} & ,\; m>1\text{ odd}\\
r_{2} & ,\; m>0\text{ even}
\end{array}\right..
\]
For $X$ defined over $K$, write $\widetilde{X_{\CC}^{\text{an}}}:=\coprod_{\sigma\in Hom(K,\CC)}\left(^{\sigma}X\right)_{\CC}^{\text{an}}$
and \[\xymatrix{H^r_{\M}\left(X,\QQ(p)\right) \ar [rd]_{\tilde{c}_{\D,\RR}^+} \ar [rr]^{\tilde{c}_{\D,\RR}} & & H^r\left( \widetilde{X^{\text{an}}_{\CC}},\RR(p)\right) \\ & H^r_{\D}\left(\widetilde{X^{\text{an}}_{\CC}},\RR(p)\right)^+ \ar @{^(->} [ur] }\]for
the map sending $\mathrm{Z}\mapsto\left(c_{\D,\RR}(^{\sigma}\mathrm{Z})\right)_{\sigma}$,
which factors through the invariants under de Rham conjugation. If
$X=\text{Spec}(K)$, then we have $H_{\D}^{1}(\widetilde{X_{\CC}^{\text{an}}},\RR(p))\cong\RR(p-1)^{\oplus d}$
and $H_{\D}^{1}(\widetilde{X_{\CC}^{\text{an}}},\RR(p))^{+}\cong\RR(p-1)^{\oplus d_{p}}$.
Write $H_{\M}^{r}(X,\RR(p))=H_{\M}^{r}(X,\QQ(p))\otimes_{\QQ}\RR$.
\begin{lem}
\label{lem4}For $X=\mathrm{Spec}(K)$, $\GG_{m,K}^{\times n}$, $(\TT_{K}^{n},S_{K}^{n})$,
or $(\TT_{K}^{n},\tilde{S}_{K}^{n})$, 
\[
\tilde{c}_{\D,\RR}^{+}\otimes\RR:\, H_{\M}^{r}(X,\RR(p))\to H_{\D}^{r}(\widetilde{X_{\CC}^{\text{an}}},\RR(p))^{+}
\]
is an isomorphism $(\forall r,p)$.\end{lem}
\begin{proof}
By \cite{Bu}, the composition
\[
K_{2p-1}(\mathcal{O}_{K})\otimes\QQ\overset{\cong}{\to}H_{\M}^{1}(\text{Spec}(K),\QQ(p))\overset{\tilde{c}_{\D,\RR}^{+}}{\to}\RR(p-1)^{\oplus d_{p}}\overset{\cdot\tfrac{2}{(2\pi\ay)^{p-1}}}{\to}\RR^{d_{p}}
\]
is exactly the Borel regulator (and the groups are zero for $r\neq1$).
The lemma follows for $X=\text{Spec}(K)$.

Let $Y$ be a smooth quasi-projective variety, defined over $K$,
and pick $p\in\GG_{m}(K)$. Write $Y\overset{\imath}{\hookrightarrow}\GG_{m,Y}\overset{\jmath}{\hookrightarrow}\mathbb{A}_{Y}^{1}\overset{\kappa}{\hookleftarrow}Y$
for the Cartesian products with $Y$ of the morphisms $\text{Spec}(K)\overset{\imath_{p}}{\hookrightarrow}\GG_{m,K}\overset{\jmath}{\hookrightarrow}\mathbb{A}_{K}^{1}\overset{\imath_{0}}{\hookleftarrow}\text{Spec}(K)$.
Then by the homotopy property, $\imath^{*}:H_{\mathcal{K}}^{r}(\GG_{m,Y},\RR(p))\to H_{\mathcal{K}}^{r}(Y,\RR(p))\cong H_{\mathcal{K}}^{r}(\mathbb{A}_{Y}^{1},\RR(p))$
splits the localization sequence
\[
\overset{\kappa_{*}}{\to}H_{\mathcal{K}}^{r}(\mathbb{A}_{Y}^{1},\RR(p))\overset{\jmath^{*}}{\to}H_{\mathcal{K}}^{r}(\GG_{m,Y},\RR(p))\overset{\Res}{\to}H_{\mathcal{K}}^{r-1}(Y,\RR(p-1))\overset{\kappa_{*}}{\to}
\]
for $\mathcal{K}=\M,\D$ (in particular, $\kappa_{*}=0$). It follows
that 
\[
H_{\mathcal{K}}^{r}(\GG_{m,Y},\RR(p))\cong H_{\mathcal{K}}^{r}(Y,\RR(p))\oplus H_{\mathcal{K}}^{r-1}(Y,\RR(p-1)),
\]
compatibly with $c_{\D,\RR}$; applying this iteratively gives the
lemma for $\GG_{m,K}^{\times n}$.

Finally, both $(\TT^{n},S_{K}^{n})$ and $(\TT_{K}^{n},\tilde{S}_{K}^{n})$
may be regarded as (co)simplicial normal crossing schemes $X^{\bullet}$.
(That is, writing $\tilde{S}_{K}^{n}=\cup Y_{i}$, we take $X^{0}=\TT_{K}^{n}$,
$X^{1}=\coprod_{i}Y_{i}$, $X^{2}=\coprod_{i<j}Y_{i}\cap Y_{j}$,
etc.) We have spectral sequences $E_{1}^{i,j}=H_{\mathcal{K}}^{2p+j}(X^{i},\RR(p))\implies H_{\mathcal{K}}^{2p+i+j}(X^{\bullet},\RR(p)),$
compatible with $c_{\D,\RR}$, and all $X^{i}$ are disjoint unions
of powers of $\GG_{m,K}$. Lemma is proved.\end{proof}
\begin{lem}
\label{lem5}The map $\imath_{n}^{*}:\, H_{\M}^{n}(\TT^{n},S^{n};\mathbb{A}(n))\to H_{\M}^{n}(\TT^{n-1},\tilde{S}^{n-1};\mathbb{A}(n))$
is zero $(\mathbb{A=\QQ\text{ or }\RR})$.\end{lem}
\begin{proof}
Form the obvious commutative square and use the results of Lemmas
\ref{lem3} and \ref{lem4}.
\end{proof}

\subsection{The Beilinson elements}

\label{secMot}To each $I\subset\{1,\ldots,n\}$ and $\epsilon:I\to\{0,\infty\}$
we associate a face map $\rho_{I}^{\epsilon}:\square^{n-|I|}\hookrightarrow\square^{n}$,
with $z_{i}=\epsilon(i)$ ($\forall i\in I$) on the image, and degeneracy
maps $\delta_{i}:\square^{n}\twoheadrightarrow\square^{n-1}$ killing
the $i^{\text{th}}$ coordinate. For any smooth quasi-projective variety
$X$ (say, over a field $K\supseteq\QQ$), let $c^{p}(X,n)$ denote
the free abelian group on subvarieties (of codimension $p$) of $X\times\square^{n}$
meeting all faces $X\times\rho_{I}^{\epsilon}(\square^{n-|I|})$ properly,
and $d^{p}(X,n)=\sum im(\text{id}_{X}\times\delta_{i}^{*})\subset c^{p}(X,n).$
Then $Z^{p}(X,\bullet):=c^{p}(X,\bullet)/d^{p}(X,\bullet)$ defines
a complex with differential 
\[
\partial=\sum_{i=1}^{n}(-1)^{i-1}\left((\text{id}_{X}\times\rho_{i}^{0})^{*}-(\text{id}_{X}\times\rho_{i}^{\infty})^{*}\right),
\]
whose $r^{\text{th}}$ homology defines Bloch's higher Chow group\begin{equation}\label{eqn2.5-1}CH^p (X,r) \cong H^{2p-r}_{\M}(X,\ZZ(p)).
\end{equation}This isomorphism does not apply for singular varieties (e.g. our simplicial
schemes above), and for our purposes in this paper it is the right-hand
side of \eqref{eqn2.5-1} that provides the correct generalization.
In particular, we have
\[
H_{\M}^{r}\left(X\times(\square^{a},\partial\square^{a}),\QQ(p)\right)\cong H_{\M}^{r-a}(X,\QQ(p))
\]
where $\partial\square^{a}:=\bigcup_{\tiny\begin{array}{c}
i\in\{1,\ldots,a\}\\
\varepsilon\in\{0,\infty\}
\end{array}}\rho_{i}^{\varepsilon}(\square^{a-1}).$ We note here that the (rational) motivic cohomology of a cosimplicial
normal-crossing scheme $X^{\bullet}$ can be computed via (the simple
complex associated to) a double complex:\begin{equation}\label{eqn2.5-2}E_0^{a,b}:=Z^p(X^a,-b)^{\#}_{\QQ} \implies H^{2p+a+b}_{\M}(X^{\bullet},\QQ(p)),
\end{equation}where ``$\#$'' denotes cycles meeting all components of all $X^{q>a}\times\partial_{I}^{\epsilon}\square^{-b}$
properly.%
\footnote{See \cite[$\S$3]{Le} and \cite[$\S$8.2]{KL} for the relevant moving
lemmas (and for detailed discussion of differentials, etc.).%
}

Continuing to write $t_{i}$ for $\tfrac{z_{i}}{z_{i}-1}$, we shall
now consider 
\[
f(\uz)=f_{n-1}(z_{1},\ldots,z_{n-1}):=1-\omega^{b}t_{1}\cdots t_{n-1}
\]
as a regular function on $\square_{\FF}^{n-1}$, and
\[
\Z:=\left\{ \left.(\uz;f(\uz),t_{1}^{N},\ldots,t_{n-1}^{N})\,\right|\,\uz\in\square^{n-1}\backslash\left|(f)_{0}\right|\right\} 
\]
as an element of
\[
\ker\left\{ Z^{n}\left(\square^{n-1}\backslash\left|(f)_{0}\right|,n\right)_{\QQ}^{\#}\overset{\partial\oplus\sum(\rho_{i}^{\varepsilon})^{*}}{\longrightarrow}\begin{array}{c}
Z^{n}(\square^{n-1}\backslash\left|(f)_{0}\right|,n-1)\bigoplus\\
\oplus_{i,\varepsilon}Z^{n}(\square^{n-2}\backslash\left|(f|_{z_{i}=\varepsilon})_{0}\right|,n)_{\QQ}
\end{array}\right\} 
\]
hence of $H_{\M}^{n}\left(\square_{\FF}^{n-1}\backslash\left|(f)_{0}\right|,\partial\square^{n-1}\backslash\partial\left|(f)_{0}\right|;\QQ(n)\right)$
(where $\partial\left|(f)_{0}\right|:=\partial\square^{n-1}\cap\left|(f)_{0}\right|=\cup_{i,\varepsilon}\left|(f|_{z_{i}=\varepsilon})_{0}\right|,$
and ``$\#$'' indicates cycles meeting faces of $\partial\square^{n-1}\backslash\partial\left|(f)_{0}\right|$
properly). The powers $t_{i}^{N}$ are unnecessary at this stage but
will be crucial later. For simplicity, we write the class of $Z$
in this group as a symbol $\{f_{n-1},t_{1}^{N},\ldots,t_{n-1}^{N}\}$.

Using Lemma \ref{lem0}, we have a (vertical) localization exact sequence\begin{equation}\label{eqn2.5-3}\small \xymatrix{{\,} \ar[d] \\ H^n_{\M}\left(\square^{n-1},\partial \square^{n-1};\QQ(n)\right) \ar @{<->} [r]^{\cong} \ar [d] & CH^n(\FF,2n-1)_{\QQ} \\ H^n_{\M}\left(\square^{n-1}\backslash\left|(f)_{0}\right|,\partial\square^{n-1}\backslash\left|(f)_{0}\right|;\QQ(n)\right) \ar [d]^{\Res_{\left|(f)_0\right|}} \\ H^{n-1}_{\M}\left(\TT^{n-2},\tilde{S}^{n-2};\QQ(n-1)\right) \ar [d] & H^{n-1}_{\M}\left(\TT^{n-1},S^{n-1};\QQ(n-1)\right) \ar [l]_{\imath^*_{n-1}} \\ {\,} } \normalsize
\end{equation}in which evidently
\[
\Res_{\left|(f)_{0}\right|}\left\{ f_{n-1},t_{1}^{N},\ldots,t_{n-1}^{N}\right\} =\imath_{n-1}^{*}\left\{ t_{1}^{N},\ldots,t_{n-1}^{N}\right\} .
\]

\begin{prop}
\label{prop2.5-1}$\Z$ lifts to a class $\tilde{\Xi}\in CH^{n}(\FF,2n-1)_{\QQ}$.\end{prop}
\begin{proof}
Apply \eqref{eqn2.5-3} and Lemma \ref{lem5}.
\end{proof}
This is essentially Beilinson's construction; we normalize the class
by 
\[
\Xi:=\frac{(-1)^{n}}{N^{n-1}}\tilde{\Xi}.
\]

\section{The higher Chow cycles\label{sec3}}

\subsection{Representing Beilinson's elements}

We first describe \eqref{eqn2.5-2} more explicitly in the relevant
cases. As above, write $\partial:\, Z^{n}(\square^{r},s)_{\QQ}^{\#}\to Z^{n}(\square^{r},s-1)_{\QQ}^{\#}$
for the higher Chow differential, and
\[
\delta:\, Z^{n}(\square^{r},s)_{\QQ}^{\#}\to\bigoplus_{i,\varepsilon}Z^{n}(\square^{r-1},s)_{\QQ}^{\#}
\]
for the cosimplicial differential $\sum_{i=1}^{r}(-1)^{i-1}\left((\rho_{i}^{0}\times\text{id}_{\square^{s}})^{*}-(\rho_{i}^{\infty}\times\text{id}_{\square^{s}})^{*}\right).$
A complex of cocycles for the top motivic cohomology group in \eqref{eqn2.5-3}
is given by $\mathfrak{Z}_{\square}^{n}(k):=$\begin{equation}\label{eqn3.1-1}Z^n_{\M}\left( (\square^{n-1}_{\FF},\partial \square^{n-1}_{\FF} ),k\right)_{\QQ} := \bigoplus_{a=0}^{n-1} \bigoplus_{\tiny \begin{array}{c}(I,\epsilon)\\|I|=a\end{array}} Z^n\left( \square^{n-a-1}_{\FF},a+k\right)_{\QQ}^{\#}
\end{equation}with differential $\mathbb{D}:=\partial+(-1)^{n-a-1}\delta$. These
are, of course, the simple complex resp. total differential associated
to the natural double complex $E_{0}^{a,b}=\bigoplus_{\tiny\begin{array}{c}
(I,\epsilon)\\
|I|=a
\end{array}}Z^{n}\left(\square_{\FF}^{n-a-1},-b\right)_{\QQ}^{\#}.$ Analogously one defines $\mathfrak{Z}_{\square\backslash f}^{n}(k):=Z_{\M}^{n}\left((\square_{\FF}^{n-1}\backslash\left|(f)_{0}\right|,\partial\square_{\FF}^{n-1}\backslash\partial\left|(f)_{0}\right|),k\right)_{\QQ}$
and $\mathfrak{Z}_{f}^{n-1}(k):=Z_{\M}^{n-1}\left((\TT^{n-2},\tilde{S}^{n-2}),k\right)_{\QQ}$
so that $\mathfrak{Z}_{f}^{n-1}(\bullet)\overset{\imath_{*}}{\to}\mathfrak{Z}_{\square}^{n}(\bullet)\to\mathfrak{Z}_{\square\backslash f}^{n}(\bullet)$
are morphisms of (homological) complexes.

Now define
\[
\theta:\,\mathfrak{Z}_{\square}^{n}(k)\to Z^{n}(\FF,n+k-1)_{\QQ}
\]
by simply adding up the cycles (with no signs) on the right-hand side
of \eqref{eqn3.1-1}. (Use the natural maps $\square^{n-a-1}\times\square^{a+k}\to\square^{n+k-1}$
obtained by concatenating coordinates.) Then we have 
\begin{lem}
\label{lem6}$\theta$ is a quasi-isomorphism of complexes.\end{lem}
\begin{proof}
Checking that $\theta$ is a morphism of complexes is easy and left
to the reader. The $a=n-1$, $(I,\epsilon)=\left(\{1,\ldots,n-1\},\underline{0}\right)$
term of \eqref{eqn3.1-1} is a copy of $Z^{n}(\FF,n+k-1)$ in $\mathfrak{Z}_{\square}^{n}(k)$
which leads to a morphism $\psi:Z^{n}(\FF,n+\bullet-1)\to\mathfrak{Z}_{\square}^{n}(\bullet)$
with $\theta\circ\psi=\text{id}$. Moreover, it is elementary that
$\psi$ is a quasi-isomorphism: taking $d_{0}=\partial$ gives 
\[
E_{1}^{a,b}=\bigoplus_{\tiny\begin{array}{c}
(I,\epsilon)\\
|I|=a
\end{array}}CH^{n}(\square_{\FF}^{n-a-1},-b)_{\QQ}\cong CH^{n}(\FF,-b)^{\oplus2^{a}{n-1 \choose a}},
\]
hence $E_{2}^{a,b}=0$ except for $E_{2}^{n-1,b}\cong CH^{n}(\FF,-b)$,
which is exactly the image of $\psi(\ker\partial)$.%
\footnote{This is true for any field, but specifically for our $\FF=\QQ(\omega)$,
the only nonzero term is $E_{2}^{n-1,n}$.%
}
\end{proof}
In particular, we may view $\theta$ as yielding the isomorphism in
the top row of \eqref{eqn2.5-1}.

By the moving lemmas of Bloch \cite{Bl2} and Levine \cite{Le}, we
have another quasi-isomorphism
\[
\frac{\mathfrak{Z}_{\square}^{n}(\bullet)}{\imath_{*}\mathfrak{Z}_{f}^{n-1}(\bullet)}\overset{\simeq}{\longrightarrow}\mathfrak{Z}_{\square\backslash f}^{n}(\bullet),
\]
which enables us to replace any $\Y_{\square\backslash f}\in\ker(\mathbb{D})\subset\mathfrak{Z}_{\square\backslash f}^{n}(n)$
by a homologous $\Y_{\square\backslash f}'$ arising as the restriction
of some $\Y_{\square}'\in\mathfrak{Z}_{\square}^{n}(n)$ with $\mathbb{D}\Y_{\square}'=\imath_{*}(\Y_{f}'')$,
$\Y_{f}''\in\ker(\mathbb{D})\in\mathfrak{Z}_{f}^{n-1}(n-1).$ This
gives an ``explicit'' prescription for computing $\Res_{\left|(f)_{0}\right|}$
in \eqref{eqn2.5-1}.

Now we come to our central point: the cycle $\Z=\{f_{n-1},t_{1}^{N},\ldots,t_{n-1}^{N}\}$
of $\S$\ref{secMot} already belongs to $\left(Z^{n}(\square_{\FF}^{n-1},n)_{\QQ}^{\#}\subseteq\right)\mathfrak{Z}_{\square}^{n}(n)$,
without ``moving'' it by a boundary. Its restriction to $\mathfrak{Z}_{\square\backslash f}^{n}(n)$
is clearly $\mathbb{D}$-closed, and $\mathbb{D}\Z=\imath_{*}\{t_{1}^{N},\ldots,t_{n-1}^{N}\}=:\imath_{*}\T$.
By Proposition \ref{prop2.5-1}, the class of $\T$ in homology of
$\mathfrak{Z}_{f}^{n-1}(\bullet)$ is trivial, and so there exists
$\T'\in\mathfrak{Z}_{f}^{n-1}(n)$ with $\mathbb{D}\T'=-\T.$ Defining
\[
\W:=\imath_{*}\T'\;,\;\;\;\;\;\tilde{\Z}:=\Z+\W\;,
\]
we now have $\mathbb{D}\tilde{\Z}=0.$ This allows us to make a rather
precise statement about the lift in Proposition \ref{prop2.5-1}.
Write $\mathsf{p}_{i}:\square^{2n-1}\twoheadrightarrow\square^{n-i}$
for the projection $(z_{1},\ldots,z_{2n-1})\mapsto(z_{1},\ldots,z_{n-i})$.
\begin{thm}
\label{thm1}$\tilde{\Xi}$ has a representative in $Z^{n}(\FF,2n-1)_{\QQ}$
of the form 
\[
\tilde{\CZ}=\CZ+\CW=\CZ+\CW_{1}+\CW_{2}+\cdots+\CW_{n-1},
\]
where $\CZ=\theta(\Z)$ {\rm{(}}i.e., $\Z$ interpreted as an element
of $Z^{n}(\FF,2n-1)_{\QQ}${\rm{)}} and $\CW_{i}$ is supported on
$\mathsf{p}_{i}^{-1}\left|(f_{n-i})_{0}\right|$.\end{thm}
\begin{proof}
Viewing $\left(\left|(f_{n-1})_{0}\right|,\partial\left|(f_{n-1})_{0}\right|\right)\cong\left(\TT^{n-2},\tilde{S}^{n-2}\right)$
as a simplicial subscheme $\mathfrak{X}^{\bullet}$ of $\left(\square^{n-1},\partial\square^{n-1}\right)=:X^{\bullet}$,
$\mathfrak{X}^{i-1}\subset X^{i-1}$ comprises $2^{i-1}{n-1 \choose i-1}$
copies of $\left|(f_{n-i})_{0}\right|\subset\square^{n-1}.$ We may
decompose 
\[
\W\in\bigoplus_{i=1}^{n}\bigoplus_{\tiny\begin{array}{c}
(I,\epsilon)\\
|I|=i-1
\end{array}}\imath_{*}Z^{n-1}\left(\left|(f_{n-i})_{0}\right|,n+i-1\right)_{\QQ}^{\#}\subset\bigoplus_{i=1}^{n-1}E_{0}^{i-1,-n-i+1}
\]
into its constituent pieces $\W_{i}\in E_{0}^{i-1,-n-i+1}$, and define
$\CW_{i}:=\theta(\W_{i})$ and $\CW=\theta(\W)$. Clearly $\text{supp}(\CW_{i})\subset\mathsf{p}_{i}^{-1}\left|(f_{n-i})_{0}\right|,$
and $\tilde{\CZ}:=\theta(\tilde{\Z})$ is $\partial$-closed, giving
the desired representation.\end{proof}
\begin{rem}
\label{rem3.1-1}In fact, $\sigma(\CZ)$ belongs to $Z_{\RR}^{n}(\text{Spec}(\CC),2n-1)_{\QQ}$
for any $\sigma\in Hom(\FF,\CC)$: the intersections $T_{z_{1}}\cap\cdots\cap T_{z_{k}}\cap(\rho_{I}^{\epsilon})^{*}\sigma(\CZ)$
are empty excepting $T_{z_{1}}\cap\cdots\cap T_{z_{k}}\cap\sigma(\CZ)$
for $k\leq n-1$ and $T_{z_{1}}\cap\cdots\cap T_{z_{k}}\cap(\rho_{n}^{0})^{*}\sigma(\CZ)$
for $k\leq n-2$, which are both of the expected real codimension.
A trivial modification of the above argument then shows that the $\CW_{i}$
may be chosen so that the $\sigma(\CW_{i})$ (and hence $\sigma(\tilde{\CZ})$)
are in $Z_{\RR}^{n}(\text{Spec}(\CC),2n-1)_{\QQ}$ as well. We shall
henceforth assume that this has been done.
\end{rem}

\subsection{Computing the KLM map}

We begin by simplifying the formula \eqref{eqn2.4-1} for the regulator
map.
\begin{lem}
\label{lem7}Let $K\subset\CC$ and suppose $Z\in\ker(\partial)\subset Z_{\RR}^{n}(\mathrm{Spec}(K),2n-1)_{\QQ}$
satisfies \begin{equation}\label{eqn3.2-1}T_{z_1}\cap\cdots\cap T_{z_n}\cap Z^{\text{an}}_{\CC} = \emptyset .
\end{equation}Then
\[
c_{\D,\QQ}(Z)=\int_{Z_{\CC}^{\text{an}}\cap T_{z_{1}}\cap\cdots\cap T_{z_{n-1}}}\log(z_{n})\tfrac{dz_{n+1}}{z_{n+1}}\wedge\cdots\wedge\tfrac{dz_{2n-1}}{z_{2n-1}}
\]
in $\CC/\QQ(n).$\end{lem}
\begin{proof}
We have $c_{\D,\QQ}(Z)=$
\[
\sum_{k=1}^{n-1}(2\pi\ay)^{k-n}\int_{Z_{\CC}^{\text{an}}}R_{2n-1}^{(k)}+\int_{Z_{\CC}^{\text{an}}}R_{2n-1}^{(n)}+\sum_{k=1}^{n-1}(2\pi\ay)^{k}\int_{Z_{\CC}^{\text{an}}}R_{2n-1}^{(n+k)}.
\]
The terms $\int_{Z_{\CC}^{\text{an}}}R_{2n-1}^{(k)}$ are zero by
type, since $\dim_{\CC}Z_{\CC}=n-1$, and the $\int_{Z_{\CC}^{\text{an}}}R_{2n-1}^{(n+k)}$
are integrals over $Z_{\CC}^{\text{an}}\cap T_{z_{1}}\cap\cdots\cap T_{z_{n+k-1}}=\emptyset$.
So only the middle term remains.\end{proof}
\begin{lem}
\label{lem8}For any $\sigma\in Hom(\FF,\CC)$, $T_{z_{1}}\cap\cdots\cap T_{z_{n}}\cap\sigma(\tilde{\CZ})=\emptyset$.\end{lem}
\begin{proof}
From Theorem \ref{thm1}, $\sigma(\CW_{i})$ is supported over $\mathsf{p}_{i}^{-1}\left(\left|(f_{n-i})_{0}\right|\right)$;
that is, on $\sigma(\CW_{i})$ we have $z_{1}\cdots z_{n-i}=\bar{\zeta}^{b}$,
and so $T_{z_{1}}\cap\cdots\cap T_{z_{n-i}}\cap\sigma(\CW_{i})=\emptyset$
since $\bar{\zeta}^{b}\notin(-1)^{n-i}\RR_{+}$. On $\sigma(\CZ)$,
$z_{n}=f_{n-1}(z_{1},\ldots,z_{n-1})=1-\zeta^{b}t_{1}\cdots t_{n-1}$
(where $t_{i}=\tfrac{z_{i}}{z_{i}-1}$); and on $T_{z_{i}}$, $t_{i}\in[0,1]$.
It follows that on $T_{z_{1}}\cap\cdots\cap T_{z_{n}}\cap\sigma(\CZ)$,
$z_{n}$ belongs to $\RR_{-}\cap\left(1-\zeta^{b}[0,1]\right),$ which
is empty.
\end{proof}
We may now compute the regulator on the cycle of Theorem \ref{thm1},
independently of the choice of the $\CW_{i}$:
\begin{thm}
\label{thm2} $c_{\D,\QQ}(\sigma(\Xi))=Li_{n}(\zeta^{b})\in\CC/\QQ(n).$\end{thm}
\begin{proof}
By Lemmas \ref{lem7} and \ref{lem8}, we obtain $c_{\D,\QQ}(\sigma(\tilde{\CZ}))=$\begin{multline*}
\int_{\sigma(\CZ)^{\text{an}}_{\CC}\cap T_{z_1}\cap\cdots\cap T_{z_{n-1}}} \log(z_n)\tfrac{dz_{n+1}}{z_{n+1}}\wedge\cdots\wedge \tfrac{dz_{2n-1}}{z_{2n-1}} \\ 
+\; \sum_{i=1}^{n-1} \int_{\sigma(\CW_i)^{\text{an}}_{\CC}\cap T_{z_1}\cap\cdots\cap T_{z_{n-1}}} \log(z_n) \tfrac{dz_{n+1}}{z_{n+1}} \wedge\cdots\wedge \tfrac{dz_{2n-1}}{z_{2n-1}}\, ,
\end{multline*}in which (by the proof of Lemma \ref{lem8}) $\sigma(\CW_{i})_{\CC}^{\text{an}}\cap T_{z_{1}}\cap\cdots\cap T_{z_{n-1}}=\emptyset$
($\forall i$). The remaining (first) term becomes
\[
\int_{\uz\in\RR_{-}^{\times(n-1)}}\log(f_{n-1}(\uz))\tfrac{dt_{1}^{N}}{t_{1}^{N}}\wedge\cdots\wedge\tfrac{dt_{n-1}^{N}}{t_{n-1}^{N}}=
\]
\[
(-N)^{n-1}\int_{\underline{t}\in[0,1]^{\times(n-1)}}\log\left(1-\zeta^{b}t_{1}\cdots t_{n-1}\right)\tfrac{dt_{1}}{t_{1}}\wedge\cdots\wedge\tfrac{dt_{n-1}}{t_{n-1}}=
\]
\[
(-N)^{n-1}\int_{0}^{\zeta^{b}}\int_{0}^{u_{n-1}}\cdots\int_{0}^{u_{2}}\log(1-u_{1})\tfrac{du_{1}}{u_{1}}\wedge\cdots\wedge\tfrac{du_{n-1}}{u_{n-1}}=
\]
\[
(-1)^{n}N^{n-1}Li_{n}(\zeta^{b}),
\]
where $u_{n-1}=\zeta^{b}t_{n-1}$, $u_{n-2}=\zeta^{b}t_{n-2}t_{n-1}$,
$\ldots$, $u_{1}=\zeta^{b}t_{1}\cdots t_{n-1}.$
\end{proof}
To write the image of our cycles under the Borel regulator, we refine
notation by writing $\sigma_{a}$ (for $\sigma:\omega\mapsto e^{\frac{2\pi\ay a}{N}}$),
$f_{n-1,b}=1-\omega^{b}t_{1}\cdots t_{n-1}$, $\Xi_{b}$, $\tilde{\CZ}_{b}$,
$\CZ_{b}$, etc. So Theorem \ref{thm2} reads $c_{\D,\QQ}(\sigma_{a}(\Xi_{b}))=Li_{n}\left(e^{\frac{2\pi\ay ab}{N}}\right),$
and one has the
\begin{cor}
\label{cor1}Let $N\geq3$ and set 
\[
A:=\left\{ a\in\mathbb{N}\,\left|\,(a,N)=1\text{ and }1\leq a\leq\left\lfloor \tfrac{N}{2}\right\rfloor \right.\right\} ;
\]
then for any $b\in A$, 
\[
\tilde{c}_{\D,\RR}^{+}(\Xi_{b})=\left(\pi_{n}(Li_{n}(e^{\frac{2\pi\ay ab}{N}})\right)_{a\in A}\in\RR(n-1)^{\oplus\frac{1}{2}\phi(N)},
\]
where $\pi_{n}:\,\CC\to\RR(n-1)$ is $\mathbf{i}\mathrm{Im}$ {\rm{[}}resp.
$\mathrm{Re}${\rm{]}} for $n$ even {\rm{[}}resp. odd{\rm{]}}.
If $N=2$, then $\tilde{c}_{\D,\RR}^{+}=0$ for $n$ even and $\tilde{c}_{\D,\RR}^{+}(\Xi_{1})=\zeta(n)\in\RR(n-1)$
for $n$ odd.
\end{cor}
As an immediate consequence, we get a (rational) basis for the higher
Chow cycles on a point over any abelian extension of $\QQ$:
\begin{cor}
\label{cor2}The $\{\Xi_{b}\}_{b\in A}$ span $CH^{n}(\FF,2n-1)_{\QQ}$.
Moreover, for any subfield $\mathbb{E}\subset\FF$, with $\Gamma=\text{Gal}(\FF/\mathbb{E})$,
there exists a subset $B\subset A$ {\rm{(}}with $|B|=d_{n}(\mathbb{E})${\rm{)}}
such that the $\left\{ \sum_{\gamma\in\Gamma}{}^{\gamma}\Xi_{b}\right\} _{b\in B}$
span $CH^{n}(\mathbb{E},2n-1)_{\QQ}.$\end{cor}
\begin{proof}
In view of Lemma \ref{lem4}, for the first statement we need only
check the linear independence of the vectors $\underline{v}^{(b)}$
in Corollary \ref{cor1}. Let $\chi$ be one of the $\tfrac{1}{2}\phi(N)$
Dirichlet characters modulo $N$ with $\chi(-1)=(-1)^{n-1}$; and
let $\rho_{\alpha}:\CC^{|A|}\to\CC^{|A|}$ be the permutation operator
defined by $\mu(\underline{v})_{j}=v_{\alpha\cdot j}$, where $\alpha\in(\ZZ/N\ZZ)^{*}$
is a generator. Then the linear combinations
\[
\underline{v}^{\chi}:=\sum_{b\in A}\chi(b)\underline{v}^{(b)}=\left(\frac{1}{2}\sum_{b=1}^{N}\chi(b)\pi_{n}\left(Li_{n}(e^{\frac{2\pi\ay ab}{N}})\right)\right)_{a\in A}
\]
are independent (over $\CC$) provided they are nonzero, since their
eigenvalues $\overline{\chi(\alpha)}$ under $\rho_{\alpha}$ are
distinct. By the computation in \cite[pp. 420-422]{Za}, if $\chi$
is induced from a primitive character $\chi_{0}$ modulo $N_{0}=N/M$,
then (with $\mu=$ M\"obius function, $\tau(\cdot)=$ Gauss sum)
\[
v_{1}^{\chi}=\frac{1}{2M^{n-1}}\left\{ \sum_{d|M}\mu(d)\chi_{0}(d)d^{n-1}\right\} \tau(\chi_{0})L(\overline{\chi_{0}},n),
\]
the last two factors of which are nonzero by primitivity of $\chi_{0}$;
the bracketed term is $\prod_{\tiny\begin{array}{c}
p>1\text{ prime}\\
p|M
\end{array}}\left(1-\chi_{0}(p)p^{n-1}\right),$ hence also nonzero.

The second statement follows at once, since the composition of $\sum_{\gamma\in\Gamma}$
with $CH^{n}(\mathbb{E},2n-1)_{\QQ}\hookrightarrow CH^{n}(\FF,2n-1)_{\QQ}$
is a multiple of the identity.
\end{proof}

\section{Explicit representatives}

We finally turn to the construction of the cycles described by Theorem
\ref{thm1}. Here the benefit of using $t_{i}^{N}$ (at least, if
one is happy to work rationally) comes to the fore: it allows us to
obtain uniform formulas for all $N$, and to use as few terms as possible;
in fact, it turns out that \emph{for all} $n$ it is possible to take
$\CW_{3}=\cdots=\CW_{n-1}=0$. (While it is easy to argue abstractly
that $\CW_{n-1}$ can always be taken to be zero, this stronger statement
surprised us.) For brevity, we shall use the notation $\left(f_{1}(\underline{t},u,v),\ldots,f_{m}(\underline{t},u,v)\right)$
for 
\[
\left\{ \left.\left(f_{1}(\underline{t},u,v),\ldots,f_{m}(\underline{t},u,v)\right)\right|t_{i},u,v\in\PP^{1}\right\} \cap\square^{m};
\]
all precycles are defined over $\FF=\QQ(\omega)$, and we write $\xi:=\omega^{b}$.

\subsection{$K_{3}$ case $(n=2)$}

Let $\CZ=\left(\tfrac{t}{t-1},1-\xi t,t^{N}\right),$ as dictated
by Theorem \ref{thm1}; then all $\partial_{i}^{\varepsilon}\CZ=0$.
In particular,
\[
\partial_{1}^{0}\CZ=\left(1-\xi t,t^{N}\right)|_{\frac{t}{t-1}=0}=(1,0)=0
\]
and
\[
\partial_{2}^{0}\CZ=\left(\tfrac{\xi^{-1}}{\xi^{-1}-1},\xi^{-N}\right)=\left(\tfrac{1}{1-\xi},1\right)=0.
\]
So we may take $\CW=0$ and $\tilde{\CZ}=\CZ$.

In contrast, if we took $\CZ=\left(\tfrac{t}{t-1},1-\xi t,t\right),$
then $\partial_{2}^{0}\CZ=\left(\tfrac{1}{1-\xi},\xi^{-1}\right)$
and a nonzero $\CW$-term is required.

\subsection{$K_{5}$ case $(n=3)$}

Of course $\CZ=\left(\tfrac{t_{1}}{t_{1}-1},\tfrac{t_{2}}{t_{2}-1},1-\xi t_{1}t_{2},t_{1}^{N},t_{2}^{N}\right).$
Taking
\[
\CW_{1}=\frac{1}{2}\left(\tfrac{t_{1}}{t_{1}-1},\tfrac{1}{1-\xi t_{1}},\tfrac{(u-t_{1}^{N})(u-t_{1}^{-N})}{(u-1)^{2}},t_{1}^{N}u,\tfrac{u}{t_{1}^{N}}\right),
\]
we note that $z_{2}=\tfrac{1}{1-\xi t_{1}}$ $\implies$ $t_{2}=\tfrac{(1-\xi t_{1})^{-1}}{(1-\xi t_{1})^{-1}-1}=\tfrac{1}{\xi t_{1}}$
$\implies$ $f_{2}(t_{1},t_{2})=0.$ Now we have 
\[
\partial\CZ=\partial_{3}^{0}\CZ=\left.\left(\tfrac{t_{1}}{t_{1}-1},\tfrac{t_{2}}{t_{2}-1},t_{1}^{N},t_{2}^{N}\right)\right|_{1-\xi t_{1}t_{2}=0}=\left(\tfrac{t_{1}}{t_{1}-1},\tfrac{1}{1-\xi t_{1}},t_{1}^{N},\tfrac{1}{t_{1}^{N}}\right)
\]
and
\[
\partial\CW_{1}=-\partial_{3}^{\infty}\CW_{1}=-2\cdot\tfrac{1}{2}\left(\tfrac{t_{1}}{t_{1}-1},\tfrac{1}{1-\xi t_{1}},t_{1}^{N},\tfrac{1}{t_{1}^{N}}\right)=-\partial\CZ.
\]
Therefore $\tilde{\CZ}=\CZ+\CW_{1}$ is closed.
\begin{rem}
\label{rem4.2-1}See \cite[$\S$3.1]{Pe} for a detailed discussion
of properties of these cycles, esp. the (integral!) distribution relations
of {[}loc. cit., Prop. 3.1.26{]}.
\end{rem}
In particular, we can specialize to $N=2$ to obtain
\[
2\tilde{\CZ}=2\left(\tfrac{t_{1}}{t_{1}-1},\tfrac{t_{2}}{t_{2}-1},1+t_{1}t_{2},t_{1}^{2},t_{2}^{2}\right)+\left(\tfrac{t_{1}}{t_{1}-1},\tfrac{1}{1+t_{1}},\tfrac{(u-t_{1}^{2})(u-t_{1}^{-2})}{(u-1)^{2}},t_{1}^{2}u,\tfrac{u}{t_{1}^{2}}\right)
\]
in $Z_{\RR}^{3}(\QQ,5)$, spanning $CH^{3}(\QQ,5)_{\QQ}\cong K_{5}(\QQ)_{\QQ}$,
with 
\[
c_{\D,\QQ}(2\tilde{\CZ})=-8Li_{3}(-1)=6\zeta(3)\in\CC/\QQ(3).
\]

\subsection{$K_{7}$ case $(n=4)$}

Set \begin{flalign*}\CZ& = \left(\tfrac{t_{1}}{t_{1}-1},\tfrac{t_{2}}{t_{2}-1},\tfrac{t_{3}}{t_{3}-1},1-\xi t_{1}t_{2}t_{3},t_{1}^{N},t_{2}^{N},t_{3}^{N}\right), \;\;\CW_1 = \tfrac{1}{2} \left( \CW_1^{(1)} + \CW_1^{(2)} \right), \\
\CW_1^{(1)} &= \left(\tfrac{t_{1}}{t_{1}-1},\tfrac{t_{2}}{t_{2}-1},\tfrac{1}{1-\xi t_1 t_2},\tfrac{(u-t_1^N)(u-t_2^N)}{(u-1)(u-t_1^N t_2^N)},\tfrac{u}{t_{1}^N},\tfrac{u}{t_{2}^{N}},\tfrac{1}{u}\right), \\
\CW_1^{(2)} &= \left(\tfrac{t_{1}}{t_{1}-1},\tfrac{t_{2}}{t_{2}-1},\tfrac{1}{1-\xi t_1 t_2},\tfrac{(u-t_1^N)(u-t_2^N)}{(u-1)(u-t_1^N t_2^N)},\tfrac{t_{1}^N}{u},\tfrac{t_{2}^{N}}{u},\tfrac{u}{t_1^N t_2^N}\right), \\
\CW_2 &= -\tfrac{1}{2} \left(  \tfrac{t_1}{t_1 -1},\tfrac{1}{1-\xi t_1} ,\tfrac{(v- t_1^N u)(v-u t_1^{-N})}{(v-u^2)(v-1)} ,\tfrac{(u-t_1^N)(u-v t_1^{-N})}{(u-v)^2} ,\tfrac{vt_1^N}{u} ,\tfrac{v}{t_1^N u} ,\tfrac{u}{v} \right) .
\end{flalign*}Direct computation shows\begin{flalign*}\partial\CZ &=-\partial_{4}^{0}\CZ=-\partial_{4}^{\infty}\CW_{1}^{(1)}=-\partial_{4}^{\infty}\CW_{1}^{(2)}, \\
\partial\CW_1 &= -\tfrac{1}{2} \partial_3^{\infty} \CW_1^{(1)} + \tfrac{1}{2}\partial_4^{\infty}\CW_1^{(1)} - \tfrac{1}{2}\partial_3^{\infty}\CW_1^{(2)} + \tfrac{1}{2}\partial_4^{\infty}\CW_1^{(2)} , \\
\partial \CW_2 &= -\partial_3^{\infty}\CW_2 = \tfrac{1}{2}\partial_3^{\infty}\CW_1^{(1)} + \tfrac{1}{2}\partial_3^{\infty}\CW_1^{(2)} ,
\end{flalign*}which sum to zero.

Alternately, we can take\begin{flalign*}
\CW_1 &= \left( \tfrac{t_1}{t_1 -1}, \tfrac{t_2}{t_2 -1},\tfrac{1}{1-\xi t_1 t_2} ,\tfrac{(u-t_1^N)(u-t_2^N)}{(u-1)(u-t_1^N t_2^N)} ,\tfrac{t_1^N}{u},\tfrac{t_2^N}{u},\tfrac{u}{t_1^N t_2^N}\right) ,\\
\CW_2 &= \left( \tfrac{t_1}{t_1 -1},\tfrac{1}{1-\xi t_1} ,\tfrac{(u-vt_1^N)(u-vt_1^{-N})}{(u-v)^2},\tfrac{vt_1^N}{u},\tfrac{v}{t_1^N u},\tfrac{u}{v},v-1 \right) .
\end{flalign*}Writing\begin{flalign*}
\CV_1 &= \left( \tfrac{t_1}{t_1 -1},\tfrac{t_2}{t_2 -1},\tfrac{1}{1-\xi t_1 t_2}, t_1^N ,t_2^N,\tfrac{1}{t_1^N t_2^N} \right) , \\
\CV_2 &= \left( \tfrac{t_1}{t_1 -1},\tfrac{1}{1-\xi t_1} ,\tfrac{(u-t_1^N)(u-t_1^{-N})}{(u-1)^2},\tfrac{t_1^N}{u},\tfrac{1}{t_1^N u}, u\right) ,
\end{flalign*}one has $\partial\CZ=-\CV_{1}$, $\partial\CW_{1}=-\CV_{2}+\CV_{1}$,
$\partial\CW_{2}=\CV_{2}$; so again $\tilde{\CZ}$ is a closed cycle.

We present the general $n$ construction next, but include the $n=5$
case as an appendix (as the authors only saw the pattern after working
out this case).

\subsection{General $n$ construction $(n\geq4)$}

To state the final result, we define
\[
\CZ:=\left(\tfrac{t_{1}}{t_{1}-1},\ldots,\tfrac{t_{n-1}}{t_{n-1}-1},1-\xi t_{1}\cdots t_{n-1},t_{1}^{N},\ldots,t_{n-1}^{N}\right),
\]
\begin{multline*}
\CW_1 := \tfrac{1}{n-3}\tilde{\CW}_1 :=  \tfrac{(-1)^{n-1}}{n-3}\times \\
\left(\tfrac{t_{1}}{t_{1}-1},\ldots,\tfrac{t_{n-2}}{t_{n-2}-1},\tfrac{1}{1-\xi t_{1}\cdots t_{n-2}},\tfrac{(u-t_{1}^{N})\cdots(u-t_{n-2}^{N})}{(u-t_{1}^{N}\cdots t_{n-2}^{N})(u-1)^{n-3}},\tfrac{t_{1}^{N}}{u},\ldots,\tfrac{t_{n-2}^{N}}{u},\tfrac{u}{t_{1}^{N}\cdots t_{n-2}^{N}}\right) ,
\end{multline*}and
\[
\CW_{2}:=\tfrac{1}{n-3}\sum_{i=1}^{n-1}(-1)^{i-1}\CW_{2}^{(i)},
\]
where for $1\leq i\leq n-2$, $\CW_{2}^{(i)}:=$\begin{multline*}
\left( \tfrac{t_{1}}{t_{1}-1},\ldots,\tfrac{t_{n-3}}{t_{n-3}-1},\tfrac{1}{1-\xi t_{1}\cdots t_{n-3}},\tfrac{(u-t_{1}^{N}v)\cdots(u-t_{n-3}^{N}v)}{(u-t_{1}^{N}\cdots t_{n-3}^{N}v)(u-v)^{n-4}}, \right. \\ \left. \tfrac{vt_{1}^{N}}{u},\ldots,\tfrac{v}{u},\ldots,\tfrac{vt_{n-3}^{N}}{u},\tfrac{u}{vt_{1}^{N}\cdots t_{n-3}^{N}},v-1 \right)
\end{multline*}(with $\tfrac{v}{u}$ occurring in the $(n+i-1)^{\text{th}}$ entry%
\footnote{That is, either before ($i=1$), after ($i=n-2$), or in the middle
of the sequence $\tfrac{vt_{1}^{N}}{u},\tfrac{vt_{2}^{N}}{u},\ldots,\tfrac{vt_{n-3}^{N}}{u}$.%
}) and $\CW_{2}^{(n-1)}:=$\begin{multline*}
\left( \tfrac{t_{1}}{t_{1}-1},\ldots,\tfrac{t_{n-3}}{t_{n-3}-1},\tfrac{1}{1-\xi t_{1}\cdots t_{n-3}},\tfrac{(u-t_{1}^{N}v)\cdots(u-t_{n-3}^{N}v)}{(u-t_{1}^{-N}\cdots t_{n-3}^{-N}v)^{-1}(u-v)^{n-2}}, \right. \\ \left. \tfrac{vt_{1}^{N}}{u},\ldots,\tfrac{vt_{n-3}^{N}}{u},\tfrac{v}{ut_{1}^{N}\cdots t_{n-3}^{N}},\tfrac{u}{v},v-1 \right) .
\end{multline*}
\begin{thm}
\label{thm3}$\tilde{\CZ}=\CZ+\CW_{1}+\CW_{2}$ yields a closed cycle,
with the properties described in Theorem \ref{thm1}. {\rm{(}}In
particular, this recovers the second $K_{7}$ construction and the
$K_{9}$ construction above, for $n=4$ and $5$.{\rm{)}}\end{thm}
\begin{proof}
Writing
\[
\CY_{0}:=\partial_{n}^{0}\CZ=\left(\tfrac{t_{1}}{t_{1}-1},\ldots,\tfrac{t_{n-2}}{t_{n-2}-1},\tfrac{1}{1-\xi t_{1}\cdots t_{n-2}},t_{1}^{N},\ldots,t_{n-2}^{N},\tfrac{1}{t_{1}^{N}\cdots t_{n-2}^{N}}\right),
\]
$\CY_{i}:=\partial_{2n-1}^{0}\CW_{2}^{(i)}$ ($i=1,\ldots,n-1$),
and $\CX_{i,j}:=\partial_{j}^{\infty}\CW_{2}^{(i)}$ ($j=1,\ldots,n-2$),
one computes that $\partial\CZ=(-1)^{n-1}\CY_{0}$, 
\[
\partial\tilde{\CW}_{1}=(-1)^{n}\partial_{n}^{\infty}\tilde{\CW}_{1}+\sum_{i=1}^{n-1}(-1)^{i}\partial_{i}^{\infty}\tilde{\CW}_{1}=(-1)^{n}(n-3)\CY_{0}+\sum_{i=1}^{n-1}(-1)^{i}\CY_{i},
\]
and $\partial\CW_{2}^{(i)}=\CY_{i}+\sum_{j=1}^{n-2}(-1)^{j}\CX_{i,j}.$
We have therefore\begin{equation}\label{eqn4.5-1}\partial \tilde{\CZ} = \frac{1}{n-3}\sum_{i=1}^{n-1}\sum_{j=1}^{n-2}(-1)^{i+j-1}\CX_{i,j},
\end{equation}and for each $i>j$ the reader will verify that $\CX_{i,j}=\CX_{j,i-1}$,
so that the terms on the right-hand side of \eqref{eqn4.5-1} cancel
in pairs.
\end{proof}

\subsection{Expected implications for torsion\label{torsion}}

One of the anticipated applications of the explicit $AJ$ maps of
\cite{KLM,KL} has been the detection of torsion in higher Chow groups.
While they provide an explicit map of complexes from $Z_{\RR}^{p}(X,\bullet)$
to the \emph{integral} Deligne cohomology complex, the fact that $Z_{\RR}^{p}(X,\bullet)\subset Z^{p}(X,\bullet)$
is only a \emph{rational} quasi-isomorphism leaves open the possibility
that a given cycle with (nontrivial) torsion KLM-image is bounded
by a precycle in the larger complex. So far, therefore, any conclusions
we can try to draw about torsion are speculative, as they depend on
the (so far) conjectural extension of the KLM map to an \emph{integrally}
quasi-isomorphic subcomplex.

Let us describe what the existence of such an extension, together
with the cycles just constructed, would yield. Let $f:\,\ZZ/N\ZZ\to\ZZ$
be a function which is zero off $(\ZZ/N\ZZ)^{*}$, with $f(-b)=(-1)^{n}f(b)$,
and write 
\[
\varepsilon_{n}:=\left\{ \begin{array}{cc}
1, & n=2\\
2, & n=3\\
n-3, & n\geq4.
\end{array}\right.
\]
Then (fixing $\sigma(\omega)=\zeta_{N}=e^{\frac{2\pi\mathbf{i}}{N}}$)
the cycle 
\[
Z_{f}^{n}(N):=\varepsilon_{n}\sum_{b=0}^{N-1}f(b)\sigma(\tilde{\mathscr{Z}}_{b})\in Z_{\RR}^{n}(\QQ(\zeta_{N}),2n-1)
\]
is integral. Working up to sign, we compute (in $\CC/\ZZ$) by Theorem
\ref{thm2}\begin{flalign*}
\tau_f^n(N) &:= \tfrac{\pm 1}{(2\pi \mathbf{i})^n} c_{\mathcal{D}}(Z_f^n(N)) \\ &= \tfrac{\pm\varepsilon_n N^{n-1}}{(2\pi \mathbf{i})^n} \sum _{b=0}^{N-1} f(b) \sum_{k\geq 1}\tfrac{\zeta_N^{kb}}{k^n} \\
&= \tfrac{\pm\varepsilon_n N^{n-1}}{2 (2\pi \mathbf{i})^n}\sum_{b=0}^{N-1} f(b) \sum_{k\in \mathbb{Z} \setminus\{0\}}\tfrac{\zeta_N^{kb}}{k^n} \\ &= \tfrac{\pm\varepsilon_n N^{n-1}}{2\cdot n!} \sum_{b=0}^{N-1} f(b) B_n(\tfrac{b}{N}), \\
\end{flalign*}which is evidently a rational number.%
\footnote{$B_{n}(x)=\sum_{j=0}^{n}{n \choose j}B_{j}x^{n-j}$ is the $n^{\text{th}}$
Bernoulli polynomial (and $\{B_{j}\}$ the Bernoulli numbers)%
} This (nonconjecturally) establishes that $Z_{f}^{n}(N)$ is torsion.
Under our working (conjectural!) hypothesis, if $\tau_{f}^{n}(N)=\pm\tfrac{A_{f}^{n}(N)}{C_{f}^{n}(N)}$
in lowest form, we may additionally conclude that the order of $Z_{f}^{n}(N)$
is a multiple of $C_{f}^{n}(N)$.

For example, taking $N=5$, $n=2$, and $f(1)=f(4)=1$, $f(2)=f(3)=0$,
we obtain $Z_{f}^{2}(5)\in Z_{\RR}^{2}(\QQ(\sqrt{5}),3)$ with $\tau_{f}^{2}(5)=\tfrac{\pm1}{120}$.
This checks out with what is known (cf. Prop. 6.9 and Remark 6.10
of \cite{Pe2}), and would make $Z_{f}^{2}(5)$ a generator of $CH^{2}(\QQ(\sqrt{5}),3)$. 

For $N=2$, $f(1)=1$, and $n=2m$ (i.e. $CH^{2m}(\QQ,4m-1)$), the
above computation simplifies to\begin{flalign*}
|\tau_f^{2m}(2)| &= \tfrac{\pm \varepsilon_{2m} 2^{2m-2}}{(2m)!} B_{2m}(\tfrac{1}{2}) \\
&= \tfrac{\pm (2m-3) (2^{2m-1}-1)}{2(2m)!} B_{2m},
\end{flalign*}which yields $\tfrac{1}{24}$, $\tfrac{7}{1440}$, $\tfrac{31}{20160}$,
$\tfrac{635}{483840}$ for $m=1,2,3,4$. It is known that $CH^{2}(\QQ,3)\cong\ZZ/24\ZZ$
{[}op. cit.{]}, but the other orders seem unexpectedly large and should
warrant further investigation.

\appendix

\section{$K_{9}$ case ($n=5$)}

Begin by writing\begin{flalign*}
\CZ &= \left( \tfrac{t_1}{t_1 - 1}, \tfrac{t_2}{t_2 - 1}, \tfrac{t_3}{t_3 - 1},\tfrac{t_4}{t_4 - 1},1 - \xi t_1 t_2 t_3 t_4,  t_1^N, t_2^N,t_3^N, t_4^N \right) , \\ 
\CW_1 &= \tfrac{1}{2}\left(\tfrac{t_1}{t_1 - 1}, \tfrac{t_2}{t_2 - 1}, \tfrac{t_3}{t_3 - 1},\tfrac{1}{1 - \xi t_1t_2t_3}, \tfrac{(u - t_1^l)(u - t_2^l)(u - t_3^l)}{(u - 1)^2(u - t_1^lt_2^lt_3^l)}, \tfrac{t_1^N}{u}, \tfrac{t_2^N}{u}, \tfrac{t_3^N}{u} ,\tfrac{u}{t^N_1 t^N_2 t^N_3}\right),\\ 
\CW_2^{(1)} &= \left(\tfrac{t_1}{t_1 - 1}, \tfrac{t_2}{t_2 - 1}, \tfrac{1}{1 - \xi t_1t_2}, \tfrac{(u - t_1^N v)(u - t_2^N v)}{(u - t_1^N t_2^N v)(u - v)}, \tfrac{v}{u}, \tfrac{t_1^N v}{u}, \tfrac{t_2^N v}{u}, \tfrac{u}{v t_1^N t_2^N}, v-1\right) ,\\
\CW_2^{(2)} &= \left(\tfrac{t_1}{t_1 - 1}, \tfrac{t_2}{t_2 - 1}, \tfrac{1}{1 - \xi t_1t_2}, \tfrac{(u - t_1^N v)(u - t_2^N v)}{(u - t_1^N t_2^N v)(u - v)}, \tfrac{v t_1^N}{u}, \tfrac{ v}{u}, \tfrac{t_2^N v}{u}, \tfrac{u}{v t_1^N t_2^N}, v-1\right) ,\\ 
\CW_2^{(3)} &= \left(\tfrac{t_1}{t_1 - 1}, \tfrac{t_2}{t_2 - 1}, \tfrac{1}{1 - \xi t_1t_2}, \tfrac{(u - t_1^N v)(u - t_2^N v)}{(u - t_1^N t_2^N v)(u - v)}, \tfrac{vt_1^N}{u}, \tfrac{ v t_2^N}{u}, \tfrac{v}{u}, \tfrac{u}{v t_1^N t_2^N}, v-1\right) ,\\ 
\CW_2^{(4)} &= \left(\tfrac{t_1}{t_1 - 1}, \tfrac{t_2}{t_2 - 1}, \tfrac{1}{1 - \xi t_1t_2}, \tfrac{(u - t_1^N v)(u - t_2^N v)}{(u - v t_1^{-N}t_2^{-N})^{-1} (u - v)^3}, \tfrac{vt_1^N}{u}, \tfrac{ v t_2^N}{u},  \tfrac{v}{u t_1^N t_2^N}, \tfrac{u}{v},  v-1 \right) ,\\ 
\CW_2 &= \tfrac{1}{2}\left( \CW_2^{(1)}-\CW_2^{(2)}+\CW_2^{(3)} - \CW_2^{(4)}\right) .
\end{flalign*} To compute the boundaries, introduce\begin{flalign*} 
\CU_1 &= \left(\tfrac{t_1}{t_1 - 1}, \tfrac{t_2}{t_2 - 1}, \tfrac{t_3}{t_3 - 1},\tfrac{1}{1 - \xi t_1 t_2 t_3} ,t_1^N , t_2^N ,t_3^N , \tfrac{1}{t_1^N t_2^N t_3^N}\right) ,\\ 
\CU_2 &= \left( \tfrac{t_1}{t_1 - 1}, \tfrac{t_2}{t_2 - 1}, \tfrac{1}{1 - \xi t_1 t_2}, \tfrac{(u - t_1^N )(u - t_2^N )}{(u - t_1^N t_2^N )(u - 1)}, \tfrac{1}{u}, \tfrac{t_1^N }{u}, \tfrac{t_2^N }{u}, \tfrac{u}{t_1^N t_2^N}\right) ,\\
\CU_3 &= \left( \tfrac{t_1}{t_1 - 1}, \tfrac{t_2}{t_2 - 1}, \tfrac{1}{1 - \xi t_1 t_2}, \tfrac{(u - t_1^N )(u - t_2^N )}{(u - t_1^N t_2^N )(u - 1)}, \tfrac{t_1^N}{u}, \tfrac{1}{u}, \tfrac{t_3^N}{u}, \tfrac{u}{ t_1^N t_2^N}\right) ,\\
\CU_4 &= \left(\tfrac{t_1}{t_1 - 1}, \tfrac{t_2}{t_2 - 1}, \tfrac{1}{1 - \xi t_1 t_2}, \tfrac{(u - t_1^N )(u - t_2^N )}{(u - t_1^N t_2^N )(u - 1)}, \tfrac{t_1^N}{u}, \tfrac{  t_2^N}{u}, \tfrac{ 1}{u}, \tfrac{u}{ t_1^N t_2^N}\right) ,\\
\CU_5 &= \left(\tfrac{t_1}{t_1 - 1}, \tfrac{t_2}{t_2 - 1}, \tfrac{1}{1 - \xi t_1 t_2}, \tfrac{(u - t_1^N )(u - t_2^N )(u - t_1^{-N} t_2^{-N})}{(u - 1)^3}, \tfrac{t_1^N}{u}, \tfrac{t_2^N}{u}, \tfrac{1}{u t_1^N t_2^N}, u \right)  
\end{flalign*}and\begin{flalign*} 
\CV_1 &= \left(\tfrac{t_1}{t_1 - 1},  \tfrac{1}{1 - \xi t_1}, \tfrac{(u - t_1^N v)(u - t_1^{-N} v)}{(u -  v)^2}, \tfrac{v}{u}, \tfrac{t_1^N v}{u}, \tfrac{v}{u t_1^N}, \tfrac{u}{v}, v-1\right), \\ 
\CV_2 &= \left(\tfrac{t_1}{t_1 - 1}, \tfrac{1}{1 - \xi t_1}, \tfrac{(u - t_1^N v)(u - t_1^{-N} v)}{(u -  v)^2}, \tfrac{vt_1^N}{u}, \tfrac{v}{u}, \tfrac{v}{t_1^N u}, \tfrac{u}{v}, v-1\right), \\ 
\CV_3 &= \left(\tfrac{t_1}{t_1 - 1}, \tfrac{1}{1 - \xi t_1}, \tfrac{(u - t_1^N v)(u - t_1^{-N} v)}{(u -  v)^2}, \tfrac{vt_1^N}{u}, \tfrac{v}{u t_1^N}, \tfrac{v}{u}, \tfrac{u}{v}, v-1\right) .
\end{flalign*}Then $\partial\CZ=\CU_{1}$, $\partial\CW_{1}=-\CU_{1}+\tfrac{1}{2}\left(-\CU_{2}+\CU_{3}-\CU_{4}+\CU_{5}\right)$,
$\partial\CW_{2}^{(1)}=-\CV_{1}+\CU_{2}$, $\partial\CW_{2}^{(2)}=-\CV_{2}+\CU_{3}$,
$\partial\CW_{2}^{(3)}=-\CV_{3}+\CU_{4}$, and $\partial\CW_{2}^{(4)}=\CU_{5}-\CV_{1}+\CV_{2}-\CV_{3}$;
and so $\tilde{\CZ}$ is closed.

As for $n=3$, we obtain a generator for $CH^{5}(\QQ,9)_{\QQ}\cong K_{9}(\QQ)_{\QQ}$
by setting $N=2$ and $\xi=-1$; the integral cycle $2\tilde{\CZ}$
has $c_{\D,\QQ}(2\tilde{\CZ})=15\zeta(5)$.

\curraddr{${}$\\
\noun{Department of Mathematics, Campus Box 1146}\\
\noun{Washington University in St. Louis}\\
\noun{St. Louis, MO} \noun{63130, USA}}

\email{${}$\\
\emph{e-mail}: matkerr@math.wustl.edu}

\curraddr{\noun{${}$}\\
\noun{Department of Mathematics, Campus Box 1146}\\
\noun{Washington University in St. Louis}\\
\noun{St. Louis, MO} \noun{63130, USA}}

\email{\emph{${}$}\\
\emph{e-mail}: yyang@math.wustl.edu}
\end{document}